\documentclass{article}
\usepackage{amsmath,amsthm}
\usepackage{amsfonts,amssymb}
\usepackage[all]{xy}

\newtheorem{thm}{Theorem}[section]
\newtheorem{cor}[thm]{Corollary}
\newtheorem{lem}[thm]{Lemma}

\newtheorem{defn}[thm]{Definition}

\theoremstyle{remark}

\newcommand{\mR}{\mathbb{R}}
\newcommand{\mC}{\mathbb{C}}
\newcommand{\mN}{\mathbb{N}}
\newcommand{\mE}{\mathbb{E}}

\newcommand{\mP}{\mathbb{P}}

\newcommand{\cM}{\mathcal{M}}

\newcommand{\cP}{\mathcal{P}}
\newcommand{\cC}{\mathcal{C}}

\newcommand{\gog}{\mathfrak{g}}
\newcommand{\goh}{\mathfrak{h}}

\begin{document}
\date{}

\title{The metaplectic Howe duality and polynomial solutions for the symplectic Dirac operator}

\author{Hendrik De Bie, Petr Somberg and Vladimir Sou{\v{c}}ek}

\maketitle

\abstract
We study various aspects of the metaplectic Howe duality realized by the Fischer decomposition 
for the metaplectic representation space of polynomials on $\mR^{2n}$ valued 
in the Segal-Shale-Weil representation. As a consequence, we
determine symplectic monogenics as the space of polynomial solutions 
of the symplectic Dirac operator $D_s$.

{\bf Key words:} Symplectic Dirac operator, Symplectic monogenics, Fischer decomposition, Howe duality.

{\bf MSC classification:}{ 30G35, 37J05, 37J05, 17B10}.
\endabstract

\section{Introduction}

The classical topic of separation of variables, realized by a Howe dual
pair acting on the representation space of interest, is one of the cornerstones in harmonic
analysis. The general abstract classification scheme of the Howe correspondence  
is formulated in \cite{howe}, \cite{tcit}.
One of the geometrically most interesting classical examples leading to the notion of 
spherical harmonic is the Howe dual pair $O(n)\times \mathfrak{sl}(2,\mR)$, 
realized by orthogonal Lie group $O(n)$ acting on the space of polynomials on $\mR^{n}$.
The refinement given by a ``double cover" of this dual pair is associated to the case of spinor-valued polynomials,
realized by 
$Pin(n) \times \mathfrak{osp}(1|2)$ and known in classical Clifford analysis 
as the Fischer decomposition for the space 
of spinor-valued polynomials, see \cite{howeBDES} for more details.
A similar situation appears also in the case of spinor-valued forms, where 
$(Pin(n),\mathfrak{sl}(2,\mR))$ can be regarded as
a ``double cover" of the classical Howe dual pair $(O(n),\mathfrak{o}(3))$ used to 
decompose the space of spinor valued differential forms on a $Spin(n)$-manifold 
into irreducible subbundles,
\cite{slup}.
Recently, a lot of attention (see e.g. \cite{H1, H4, H5, H8}) was devoted to the theory of harmonic and Clifford analysis  
on superspace, leading to dual pairs 
$Sp(2n,\mR) \times \mathfrak{sl}(2,\mR)$ and $Sp(2n,\mR) \times \mathfrak{osp}(1|2)$.

In the present article we start with the polynomial algebra 
on an even dimensional symplectic vector space 
$(\mR^{2n},\omega)$ equipped with canonical symplectic form 
$\omega\in\wedge^2({\mR^{2n}})^\star$. 
The analogue of the spinor representation in this situation
was described many years ago by B. Kostant, who introduced for 
the purposes of representation theory and 
geometric quantization a symplectic analogue of the Dirac operator 
called the symplectic Dirac operator $D_s$, 
see \cite{KOS}. 
The symplectic Dirac operator was studied mainly from the 
geometrical point of view, see \cite{MR2252919} and the references 
therein, and also as an invariant differential operator in \cite{KAD}. 
The spectral properties of the symplectic Dirac 
operator are difficult to obtain and as for its kernel, basically 
nothing is known up to now. 

On the other hand, the general abstract algebraic classification 
scheme of reductive dual pairs \cite{tcit} shows
the existence of the dual pair 
$\mathfrak{sp}(2n,\mR)\times \mathfrak{sl}(2,\mR)\subset \mathfrak{sp}(6n,\mR)$, 
responsible for the multiplicity
free decomposition of the space of polynomial symplectic spinors. However, its natural 
geometrical model comprising the intrinsic action of the metaplectic lift of this dual pair was not constructed.
Let us emphasize that some of the generators of the Howe dual partner (in our case, the symplectic Dirac operator)
are the starting point in the geometric analysis of geometrical structures on manifolds (in our case, the 
space of symplectic spinors on the symplectic space $(\mR^{2n},\omega)$.)

The aim of our paper is to fill this gap by describing the full analogue of the Fischer decomposition
for polynomials on the symplectic vector space with values in the vector space of Kostant's spinors,
naturally including the symplectic Dirac operator and symplectic Clifford algebras. 
  Inspired by the terminology used for the Dirac operator in $Spin$-geometry, we call 
symplectic monogenics the particular irreducible pieces in the solution space of the 
symplectic Dirac operator.
In addition, our geometric realization leads to the dual pair 
 $(Mp(2n,\mR),\mathfrak{sl}(2,\mR))$ as a ``double cover" of the classical 
Howe dual pair $(Sp(2n,\mR), \mathfrak{so}(2,1))$
in the space of endomorphisms of polynomials on $\mR^{2n}$ valued in the 
Segal-Shale-Weil representation of the 
metaplectic Lie group $Mp(2n,\mR)$. Following general principles of Howe dual pairs, 
this is the underlying structure responsible for
the Fischer decomposition mentioned above.

It follows from general abstract principles that the existence of the Howe 
dual pair $G_1\times\gog_2$ allows to separate variables
(in certain specific situations termed Fischer decomposition). This amounts 
to express the elements of the representation 
space as $\sum_{i\in I}R_i\otimes S_i$, where $I$ is an indexing set depending 
on the representation
and $R_i$ resp. $S_i$ is an irreducible representation of $G_1$ resp. $\gog_2$. 
In our specific case $G_1=Mp(2n,\mR)$,
$\gog_2 =\mathfrak{sl}(2,\mR)$ and the representation on polynomials on the 
symplectic vector space valued in 
the Segal-Shale-Weil representation
decomposes into non-isomorphic irreducible infinite dimensional simple modules for 
$\gog_2 =\mathfrak{sl}(2,\mR)$. We note that there are several conventions for 
the Segal-Shale-Weil representation, realized as a direct sum of the pair of 
highest weight or lowest weight modules (or, their completions on the space
of Schwartz functions, for example.)
  
The paper is organized as follows. In Section 2 and Section 3 we review some well-known 
facts on the symplectic Lie algebras and their 
finite dimensional, resp. simple pointed infinite dimensional representations. 
In Section 4 we define the symplectic Dirac operator and 
show how it appears in the realization of the dual partner $\mathfrak{sl}(2,\mR)$. In 
Section 5 we obtain the Fischer decomposition and construct 
explicit projection operators on all summands together with the consequences 
for the kernel of the symplectic Dirac operator 
on $(\mR^{2n},\omega )$ with $\omega$ the canonical symplectic form on $\mR^{2n}$. 
We end with some conclusions and an outlook 
for further research.

Throughout the article $\mN_0$ denotes the set of natural number including zero and
$\mN$ natural numbers without zero.

\section{Symplectic Lie algebra, symplectic Clifford algebra and simple 
weight modules for $\mathfrak{sp}(2n,\mR)$}
\label{sec_2}

In this section we review several basic facts related to the 
structure of the simple Lie algebra $\mathfrak{sp}(2n,\mR)$ and its representation
theory, see e.g., \cite{fh}, and also symplectic Clifford algebras, 
see e.g., \cite{crum}, \cite{MR2252919}, \cite{KAD}.

Let $\epsilon_1,\dots ,\epsilon_n$ be the vectors of the canonical basis of 
$\mC^n$. Identify $\mC^n$ with the dual of the Cartan subalgebra $\goh^\star$ such 
that the root system of $\mathfrak{sp}(2n,\mR)$ is 
$$
\{
\pm(\epsilon_i\pm\epsilon_j): 1\leq i<j\leq n
\}
\cup
\{\pm 2\epsilon_i:1\leq i\leq n
\},
$$
with the set of simple roots  
$$
\triangle =\{\alpha_1=\epsilon_1-\epsilon_2, \alpha_2=\epsilon_2-\epsilon_3, \dots , \alpha_{n-1}=\epsilon_{n-1}-\epsilon_n, \alpha_n=2\epsilon_n\}
$$
and the fundamental weights $\omega_1,\dots ,\omega_n$.

Let us consider the symplectic vector space $(\mR^{2n},\omega )$ and a 
symplectic basis 
$e_1,\dots ,e_n$, $f_1,\dots ,f_n$ 
for the non-degenerate canonical two form $\omega$ on $\mR^{2n}$.
Let $E_{i,j}$ be the $2n\times 2n$ matrix with $1$ on the intersection 
of $i$-th row and $j$-th column, and zero otherwise. 
Then the symplectic Lie algebra $\mathfrak{sp}(2n,\mR)$ is given by linear span of
\begin{eqnarray*}
X_{ij}=E_{i,j}-E_{n+j,n+i}, \; Y_{ij}=E_{i,n+j}+E_{j,n+i}, \; Z_{ij}=E_{n+i,j}+E_{n+j,i}
\end{eqnarray*}
for $i,j\in\{1,\dots ,n\}$.

The metaplectic Lie algebra $\mathfrak{mp}(2n,\mR)$ is a Lie algebra attached to 
the twofold covering 
$\rho : Mp(2n,\mR)\to Sp(2n,\mR)$ of the symplectic Lie group $Sp(2n,\mR)$. It can be realized 
by homogeneity two elements in 
the symplectic Clifford algebra $Cl_s(\mR^{2n},\omega)$, where the homomorphism 
$\rho_\star : \mathfrak{mp}(2n,\mR)\to \mathfrak{sp}(2n,\mR)$ 
is given by
\begin{eqnarray}
& & \rho_\star (e_ie_j)=-Y_{ij},
\nonumber \\
& & \rho_\star (f_{i}f_{j})=Z_{ij},
\nonumber \\
& & \rho_\star (e_if_{j}+f_{j}e_i)=2X_{ij}
\end{eqnarray}
for $i,j\in\{1,\dots ,n\}$.
Recall that $Cl_s(\mR^{2n},\omega)$ is an associative unital algebra,
realized as a quotient of the tensor algebra $T(e_1,\dots ,e_n,f_1,\dots ,f_n)$ by a
two-sided ideal $I\subset T(e_1,\dots ,e_n,f_1,\dots ,f_n)$ generated by
$$
v_i\cdot v_j-v_j\cdot v_i=-2\omega (v_i,v_j)
$$ 
for all $v_i,v_j\in\mR^{2n}$.

There is another useful realization of the symplectic Lie algebra as a subalgebra 
of the Weyl algebra $W_n$ of rank $n$. Let $q_i$, $i\in\{1,\dots ,n\}$, be the generators of 
the algebra of polynomials. The Weyl algebra is an associative algebra generated by 
$\{q_i,\partial_i\}$, $i\in\{1,\dots ,n\}$, the multiplication operators $q_i$ and 
the partial differentiation with respect to $q_i$, acting on polynomials on $\mR^n$
in the variables $q_1,\ldots ,q_n$. The 
root spaces of $\mathfrak{sp}(2n,\mR)$ corresponding to simple positive roots $\alpha_i$ are 
spanned by  $q_{i+1}\partial_{i}$, $i\in\{1,\dots ,n-1\}$, and $\alpha_n$ is 
spanned by $-\frac{1}{2}\partial_n^2$.

The Segal-Shale-Weil representation $\cC$ is the direct sum of two unitary simple 
representations of $Mp(2n,\mR)$ on the vector space $L^2(\mR^{n},d\mu)$, where 
$d\mu =\exp^{-||q||^2} dq_{\mR^{n}}$ with $dq_{\mR^{n}}$ the Lebesgue measure 
on $\mR^{n}$. We take for the basis of $L^2(\mR^{n},d\mu)$ the
space of polynomials on a maximally isotropic subspace $\mR^{n}\subset\mR^{2n}$. 
The differential $L_\star :\mathfrak{mp}(2n,\mR)\to \mathrm{End}(\mathrm{Pol}(\mR^n))$ 
of the action on the Segal-Shale-Weil representation is
\begin{eqnarray}
 && L_\star (e_ie_j) = iq_iq_j\quad\mbox{for}\quad i\not= j,\quad\quad 
    L_\star (e_ie_i) = -\frac{i}{2}q_i^2,
\nonumber \\
 && L_\star (f_{i}f_{j})=i\partial_i\partial_j\quad\mbox{for}\quad i\not= j,\quad\quad 
    L_\star (f_if_i) = -\frac{i}{2}\partial_i^2,
\nonumber \\
 && L_\star (e_if_{j}+f_{j}e_i)=q_j\partial_i+\frac{1}{2}\delta_{ij}
\end{eqnarray}
for $i,j\in\{1,\dots ,n\}$.

The class of finite dimensional irreducible representations of $\mathfrak{sp}(2n,\mR)$ 
coming out 
of the decomposition of $\mathrm{Pol}(\mR^{2n})$ is given by 
symmetric powers $\mathrm{S}^i(\mC^{2n}), i\in\mN_0$, of 
the complexification of the fundamental vector representation $\mR^{2n}$. 
In particular, the $\mathfrak{sp}(2n,\mR)$-module
$\mathrm{S}^i(\mC^{2n})$ is simple with the highest weight 
$i\epsilon_1$, $i\in\mN_0$, see e.g., \cite{fh}.


\section{Decomposition of tensor products of finite dimensional representations and the Segal-Shale-Weil
representation}

Let $\gog$ be a finite dimensional simple complex Lie algebra over $\mC$, $\goh$ its Cartan 
subalgebra and $M$ a simple $\goh$-diagonalizable $\gog$-module having a weight space decomposition
$M=\bigoplus_{\mu\in Weight(M)}M_\mu$ with $Weight(M)\subset\goh^*$ denoting the set of 
weights of $M$. The module $M$ is said to have bounded multiplicities provided there is a natural
number $c\in\mN$ such that $dim(M_\mu)\leq c$ for all $\mu\in Weight(M)$. In this case, the minimal 
value of $c$ is called the degree of the module $M$. Modules of degree $1$ are called completely 
pointed.

In this section we will make explicit, for the purposes of our article, several 
results in \cite{BL} on the decomposition of the tensor product of completely 
pointed highest weight modules with a certain class of 
finite dimensional representations. In particular, we will consider the two 
simple components of the Segal-Shale-Weil representation realized as highest 
weight modules and the
symmetric powers of the fundamental vector representation 
$\mC^{2n}$ of $\gog=\mathfrak{sp}(2n,\mR)$.
Throughout the article, $V(\mu)$ denotes the highest weight module 
with highest weight $\mu$ and
$L(\mu)$ denotes the simple module of highest weight $\mu$, i.e., the quotient of $V(\mu)$ 
by its unique maximal submodule $I(\mu)\subset V(\mu)$.

Let us introduce the set 
\[
\tau^{i}=\{\sum_{j=1}^n d_j \epsilon_j \,|\, (d_j+\delta_{1,i}\delta_{n,j})\in\mN_0,\, 
\sum_{j=1}^n d_{j}=0  \,\, \mbox{mod}\,\, 2\}.
\]
Here $d_j\in\mN_0$ with $i=0$ for $L(-\frac{1}{2}\omega_n)$ and $i=1$ for 
$L(\omega_{n-1}-\frac{3}{2}\omega_n)$, respectively.
Let $\lambda=\sum_{i=1}^{n}\lambda_i\omega_i$ be a dominant integral weight written in 
the basis of fundamental weights $\omega_1,\ldots ,\omega_n$ of $\mathfrak{sp}(2n,\mR)$. 
Define the set of weights 
\begin{eqnarray*}
\tau_\lambda^{i}&=&\{\mu|\lambda-\mu=\sum_{j=1}^nd_j\epsilon_j\in\tau^i,\\
&& 0\leq d_j\leq \lambda_j\, (j=1,\dots ,n-1),
0\leq d_n+\delta_{1,i}\leq 2\lambda_n+1\}.
\end{eqnarray*}
Let us recall a result of \cite{BL} :
\begin{thm}
Let $L(-\frac{1}{2}\omega_n)$ resp. $L(\omega_{n-1}-\frac{3}{2}\omega_n)$ denote the 
simple highest weight modules corresponding to the two irreducible components of the Segal-Shale-Weil
representation. Then for any finite dimensional irreducible representation $F(\lambda)$
with highest weight $\lambda$, we have 
\begin{eqnarray}
L(-\frac{1}{2}\omega_n)\otimes F(\lambda)\simeq \bigoplus_{\mu\in \tau_\lambda^{0}}L(-\frac{1}{2}\omega_n+\mu)
\end{eqnarray}
and 
\begin{eqnarray}
L(\omega_{n-1}-\frac{3}{2}\omega_n)\otimes F(\lambda)\simeq \bigoplus_{\mu\in \tau_\lambda^{1}}L(\omega_{n-1}-\frac{3}{2}\omega_n+\mu).
\end{eqnarray}
In particular, the tensor product is completely reducible and
the decomposition is direct.
\end{thm}
Recall that the simple modules appearing in this theorem have 
no singular vectors other than the highest weight ones.  

The consequence of this result is the decomposition of the tensor product 
of $L(-\frac{1}{2}\omega_n)$ resp. $L(\omega_{n-1}-\frac{3}{2}\omega_n)$ with 
symmetric powers $\mathrm{S}^k(\mC^{2n})$, $k\in\mN_0$, of the fundamental vector 
representation $\mC^{2n}$ of $\mathfrak{sp}(2n,\mR)$. Note that these are irreducible 
representations, as mentioned in Section \ref{sec_2}.

\begin{cor}\label{decomposition}
We have for $L(-\frac{1}{2}\omega_n)$
\begin{enumerate}
\item 
In the even case $k=2l$ ($2l+1$ terms on the right-hand side): 
\begin{eqnarray*}
L(-\frac{1}{2}\omega_n)\otimes \mathrm{S}^k(\mC^{2n}) & \simeq & 
L(-\frac{1}{2}\omega_n)\oplus L(\omega_1+\omega_{n-1}-\frac{3}{2}\omega_n) \nonumber\\&&
\oplus L(2\omega_1-\frac{1}{2}\omega_n)\oplus L(3\omega_1+\omega_{n-1}-\frac{3}{2}\omega_n)
\oplus\dots
\nonumber \\
& & \oplus L((2l-1)\omega_1+\omega_{n-1}-\frac{3}{2}\omega_n)\oplus
L(2l\omega_1-\frac{1}{2}\omega_n),
\end{eqnarray*}

\item
In the odd case $k=2l+1$ ($2l+2$ terms on the right-hand side):
\begin{eqnarray*}
L(-\frac{1}{2}\omega_n)\otimes \mathrm{S}^k(\mC^{2n}) & \simeq & 
L(\omega_{n-1}-\frac{3}{2}\omega_n)\oplus L(\omega_1-\frac{1}{2}\omega_n)
\\
&&\oplus L(2\omega_1+\omega_{n-1}-\frac{3}{2}\omega_n) \oplus 
L(3\omega_1-\frac{1}{2}\omega_n)
\oplus\dots
\nonumber \\
& & \oplus L(2l\omega_1+\omega_{n-1}-\frac{3}{2}\omega_n)\oplus
L((2l+1)\omega_1-\frac{1}{2}\omega_n),
\end{eqnarray*}
\end{enumerate}
We have for $L(\omega_{n-1}-\frac{3}{2}\omega_n)$
\begin{enumerate}
\item 
In the even case $k=2l$ ($2l+1$ terms on the right-hand side): 
\begin{eqnarray*}
L(\omega_{n-1}-\frac{3}{2}\omega_n)\otimes \mathrm{S}^k(\mC^{2n}) & \simeq & 
L(\omega_{n-1}-\frac{3}{2}\omega_n)\oplus L(\omega_1-\frac{1}{2}\omega_n)
\\ && \oplus L(2\omega_1+\omega_{n-1}-\frac{3}{2}\omega_n)
\oplus\dots
\nonumber \\
& & \oplus L((2l-1)\omega_1-\frac{1}{2}\omega_n)\oplus
L(2l\omega_1+\omega_{n-1}-\frac{3}{2}\omega_n),
\end{eqnarray*}

\item
In the odd case $k=2l+1$ ($2l+2$ terms on the right-hand side):
\begin{eqnarray*}
L(\omega_{n-1}-\frac{3}{2}\omega_n)\otimes \mathrm{S}^k(\mC^{2n}) & \simeq & 
L(-\frac{1}{2}\omega_n)\oplus L(\omega_1+\omega_{n-1}-\frac{3}{2}\omega_n)
\oplus\dots
\nonumber \\
& & \oplus L(2l\omega_1-\frac{1}{2}\omega_n)\oplus
L((2l+1)\omega_1+\omega_{n-1}-\frac{3}{2}\omega_n).
\end{eqnarray*}
\end{enumerate}
\end{cor}

As we shall see, a geometrical reformulation of this Corollary in the language of 
differentials operators leads to Theorem \ref{SymplecticFischer}.


\section{Generators of the Howe dual Lie algebra $\mathfrak{sl}(2,\mR)$}
\label{sl2}
Let $(\mR^{2n},\omega)$ be the symplectic vector space with coordinates $x_1,\dots ,x_{2n}$,
coordinate vector fields $\partial_1,\dots ,\partial_{2n}$ and a
symplectic basis $e_1,f_1,\dots ,e_n,f_n$, i.e., 
$$\omega(e_i,e_j)=0, \qquad \omega(f_i,f_j)=0, \qquad \omega(e_i,f_j)=\delta_{ij}$$ 
for all $i,j=1,\dots ,n$. It follows from the action of 
$\mathfrak{sp}(2n,\mR)$ on these vectors that 
\begin{eqnarray}
& & X_s:=\sum_{j=1}^{n} (x_{2j-1} f_{j} + x_{2j}e_{j}),\nonumber\\
& & D_s:=\sum_{j=1}^{n} (\partial_{x_{2j-1}} e_{j} -  \partial_{x_{2j}}f_{j}),\nonumber\\
& & \mE:=\sum_{j=1}^{2n} x_{j}\partial_{x_{j}}
\end{eqnarray} 
are invariant and so will be used as linear maps intertwining the 
$\mathfrak{sp}(2n,\mR)$ action on 
the space $\cP \otimes \cC$ of polynomials on $\mC^{2n}$ valued in the Segal-Shale-Weil 
representation $\cC$, i.e. 
$\cP: = \mathrm{Pol}(\mC^{2n})$ 
and $\cC: = L(-\frac{1}{2}\omega_n)\oplus L(\omega_{n-1}-\frac{3}{2}\omega_n)$. 
The space of homogeneous polynomials of degree $k$ will be denoted by $\cP_{k}$. 
The operator $D_s$
is crucial for the sequel and we call it the symplectic Dirac operator.

It is easy to verify that these operators fulfill the (unnormalized)
$\mathfrak{sl}(2,\mR)$-commutation relations:
\begin{eqnarray}
& & [\mE, D_s]=-D_s,
\nonumber \\
\label{slRels}
& & [\mE, X_s]=X_s,\\
& & [D_s, X_s]=\mE +n.\nonumber 
\end{eqnarray}
The action of $\mathfrak{sp}(2n,\mR)\times \mathfrak{sl}(2,\mR)$ will generate the multiplicity free decomposition
of the representation $\cP \otimes \cC$.

Further we introduce the operator
\begin{equation}
\Gamma_{s} = X_{s} D_{s} - \frac{1}{2} \mE(2n-1 + \mE),
\label{gammaop}
\end{equation}
which is the Casimir operator in $\mathfrak{sl}(2,\mR)$. 
Using \eqref{slRels} it is easy to check that $\Gamma_{s}$ commutes with both $X_{s}$ and $D_{s}$.


\section{Fischer decomposition and homomorphisms of $\mathfrak{sp}(2n,\mR)$-modules appearing 
         in the decomposition of polynomial symplectic spinors}

Before introducing the scheme in its full generality, we start with a few explicit remarks 
concerning the homogeneities zero and one parts in the decomposition. 
The tensor product $L(-\frac{1}{2}\omega_n)\otimes \mC^{2n}$
(analogously, one can consider $L(\omega_{n-1}-\frac{3}{2}\omega_n)\otimes \mC^{2n}$) 
decomposes as a direct sum $V_{1} \oplus V_{2}$ of two invariant subspaces, given by
\begin{eqnarray} 
& & V_1:=\{\sum_{i=1}^{n}e_i s\otimes f_i -\sum_{i=1}^{n}f_i s\otimes e_i|\,\, s\in L(-\frac{1}{2}\omega_n)\},
\nonumber \\
\nonumber
& & V_2:=\{\sum_{i=1}^{n}s_i\otimes e_i+\sum_{j=1}^{n}s_j\otimes f_j;\,\, s_i,s_j\in L(-\frac{1}{2}\omega_n)|\,\,\sum_{i=1}^{n}e_is_i
+\sum_{j=1}^{n}f_js_j =0\}.
\end{eqnarray}
The map 
\begin{eqnarray}
i:& & L(\omega_{n-1}-\frac{3}{2}\omega_n)\to L(-\frac{1}{2}\omega_n)\otimes \mC^n,
\nonumber \\
& & s\mapsto \sum_{i=1}^{n}e_i s\otimes f_i -\sum_{i=1}^{n}f_i s\otimes e_i,
\end{eqnarray}
(resp. $L(-\frac{1}{2}\omega_n)\to L(\omega_{n-1}-\frac{3}{2}\omega_n)\otimes \mC^n$) is 
injective and onto $V_1$. The reason is that the injectivity $i(s)=0$ is equivalent to $e_is=0$ 
resp. $f_is=0$  for all
$i\in\{1,\dots ,n\}$. Now the symplectic Clifford algebra
relation $e_if_{j}-f_{j}e_i=\delta_{ij}$ implies $s=0$ and the result follows. In other words,
the action of $X_s$ induces an isomorphism between two irreducible submodules in the homogeneity 
zero and one. Notice that $e_i, f_i$ act in the representation space
in a consistent way with the metaplectic action, i.e., the Segal-Shale-Weil representation 
extends to an irreducible representation of the semi-direct product of the metaplectic Lie algebra  
and the Heisenberg algebra generated by $e_i, f_i$, $i=1,\dots, n$.

Now we shall apply the tool of representation theory called the infinitesimal character.
The sum of the fundamental weights (or half of the sum of positive roots) for $\mathfrak{sp}(2n,\mR)$
is $\delta =(n,n-1,\dots ,2,1)$. The highest weights of simple 
$\mathfrak{sp}(2n,\mR)$-modules, 
coming from the decomposition of the tensor products of our interest, were determined 
for each homogeneity $k\in\mN_0$ in Corollary \ref{decomposition}. The 
multiplication by $X_s$ gives an intertwining map between neighboring homogeneities, say the
$k$-th  and $(k+1)$-th. Let us determine possible target modules when restricting the
action of $X_s$ to a given simple irreducible 
$\mathfrak{sp}(2n,\mR)$-module $L(a\omega_1-\frac{1}{2}\omega_{n})$ 
with highest weight $a\omega_1-\frac{1}{2}\omega_{n}$ for some $a\leq k$ (the case of 
$L(b\omega_1+\omega_{n-1}-\frac{3}{2}\omega_{n})$ being completely analogous). The comparison of 
infinitesimal characters of the collection of weights 
$\{\mu_a=a\omega_1-\frac{1}{2}\omega_{n},\nu_b=b\omega_1+\omega_{n-1}-\frac{3}{2}\omega_{n}\}$ 
($a,b\in\mN_0$) yields that 
$$||\mu_a+\delta||^2=||\nu_b+\delta||^2$$ 
if and only if either
\begin{enumerate}
\item
$2a+n-\frac{1}{2}=2b+n-\frac{1}{2}$, which implies $a=b$, or
\item
$2a+n-\frac{1}{2}=-(2b+n-\frac{1}{2})$, i.e., $a+b=-n+\frac{1}{2}$
and there is no solution in this case.
\end{enumerate}
It remains to prove that the image of $X_s$, when restricted to a 
simple module in the $k$-th homogeneity, is nonzero (or, as follows from the irreducibility,
is the simple module in the $(k+1)$-th column with the same infinitesimal 
character.)  

To complete this line of reasoning, we employ the Lie algebra $\mathfrak{sl}(2,\mR)$ from 
Section \ref{sl2}. To illustrate its impact explicitly, we start in homogeneity 
zero with the simple module $L(-\frac{1}{2}\omega_n)$ and
assume that it is in $\mathrm{Ker}(X_s)$. Because it is in $\mathrm{Ker}(D_s)$, it is in the 
kernel of the commutator $[D_s,X_s]=\mE +n$. However, $\mE +n$ acts in the homogeneity
zero by $n$, which is the required contradiction and so $X_s$ acts as an 
isomorphism $L(-\frac{1}{2}\omega_n)\to L(\omega_{n-1}-\frac{1}{2}\omega_n)$.
Let us now consider the action of $X_s$ on $L(\omega_{n-1}-\frac{1}{2}\omega_n)$
sitting in the homogeneity one part, and assume it acts trivially. Then due 
to the previous isomorphism, this kernel is $\mathrm{Ker}(X_s^2)$ when $X_s^2$ is acting on   
$L(-\frac{1}{2}\omega_n)$ in homogeneity zero. As before, the commutator 
$[X_s^2,D_s]$ acts by zero. Because it is equal to $-X_s (\mE+n)-(\mE+n)  X_s$, it acts
on homogeneity zero elements by $-(2n+1)X_s$ and due to the fact that $X_s$
is an isomorphism, it is nonzero and so yields the contradiction. In conclusion,
$X_s:L(\omega_{n-1}-\frac{1}{2}\omega_n)\to L(-\frac{1}{2}\omega_n)$ acting
between homogeneity one and two elements is an isomorphism. Clearly, one can iterate 
the procedure further using the subsequent Lemma \ref{ActSympl} and 
Lemma \ref{calcSympDirac} on $\mathfrak{sp}(2n,\mR)$-invariant
intertwining operators acting on the direct sum of simple highest weight 
$\mathfrak{sp}(2n,\mR)$-modules. An analogous induction procedure can be used to prove the 
isomorphic action of 
$D_s$.

In what follows, we turn the previous qualitative observation into 
a more quantitative statement. Following the decomposition of 
$\cP_{l} \otimes \cC$ in Corollary \ref{decomposition}, we first 
introduce the concept of symplectic monogenic polynomial:
\begin{defn}
Denote by ${\fam2 M}^+_l$ resp. ${\fam2 M}^-_l$ the 
simple $\mathfrak{sp}(2n,\mR)$-modules 
with highest weight $L(l\omega_1-\frac{1}{2}\omega_n)$ 
resp. $L(l\omega_1+\omega_{n-1}-\frac{3}{2}\omega_n))$,
and call them {\bf symplectic monogenics} of degree $l$ 
(or $l$-homogeneous symplectic monogenics). 
\end{defn}

We put 
${\fam2 M}_l:={\fam2 M}^+_l\oplus {\fam2 M}^-_l$. Based on our previous discussion, this space is characterized by
\[
\cM_{l} = \mathrm{Ker}{(D_{s})} \cap (\cP_{l} \otimes \cC)
\]
and the name symplectic monogenic is hence justified by analogy with the orthogonal 
case (see e.g., \cite{howeBDES, Orsted}).

We then obtain two auxiliary lemmas.

\begin{lem}
\label{ActSympl}
Suppose $M_{\ell} \in \cM_{\ell}$ is a symplectic monogenic of degree $\ell$. Then
\[
D_{s} (X_{s}^{k}M_{\ell}) = \frac{1}{2}k(2n+2l+k-1)X_{s}^{k-1}M_{\ell}.
\]
\end{lem}

\begin{proof}
A straightforward proof follows by induction.
\end{proof}

\begin{lem}
\label{calcSympDirac}
Suppose $M_{\ell} \in \cM_{\ell}$ is a symplectic monogenic of degree $\ell$. Then
\[
D_{s}^{j} (X_{s}^{k}M_{\ell}) = c_{j,k,\ell} X_{s}^{k-j}M_{\ell}
\]
with
\[
c_{j,k,l} = \left\{ \begin{array}{ll}\dfrac{1}{2^{j}} \dfrac{k!}{(k-j)!} \dfrac{(2n+2l+k-1)!}{(2n+2l+k-j-1)!}&j \leq k,\\0&j>k. \end{array}\right.
\]
\end{lem}

\begin{proof}
The lemma follows from $j$ iterations of Lemma \ref{ActSympl}.
\end{proof}

The previous considerations can be summarized in the symplectic analogue of the classical theorem 
on separation of variables in the orthogonal case, see for example \cite{Orsted} and 
the references therein.
\begin{thm} 
\label{SymplecticFischer}
The space $\cP\otimes \cC$ decomposes under the action of 
$\mathfrak{sl}(2,\mR)$ into the direct sum of simple
weight $\mathfrak{sp}(2n,\mR)$-modules
\[
\bigoplus_{l=0}^\infty\bigoplus_{j=0}^\infty X_s^j{\fam2 M}_l,
\]
where we used the notation ${\fam2 M}_l:={\fam2 M}^+_l\oplus {\fam2 M}^-_l$. The decomposition takes
the form of an infinite triangle \[
\xymatrix@=11pt{\cP_0 \otimes \cC  &  \cP_1 \otimes \cC & \cP_2 \otimes \cC & \cP_3 \otimes \cC \ar@{=}[d] & \cP_4 \otimes \cC & \cP_5 \otimes \cC  &\ldots \\
\cM_0 \ar[r] & X_s \cM_0 \ar[r] & X_s^2 \cM_0 \ar[r] & X_s^3 \cM_0 \ar @{} [d] |{\oplus}
 \ar[r] & X_s^4 \cM_0 \ar[r] & X_s^5 \cM_0 &\ldots\\
&\cM_1 \ar[r] & X_s \cM_1 \ar[r] & X_s^2 \cM_1 \ar @{} [d] |{\oplus}
 \ar[r] & X_s^3 \cM_1 \ar[r] & X_s^4 \cM_1  &\ldots\\
&&\cM_2 \ar[r] & X_s \cM_2 \ar @{} [d] |{\oplus}
 \ar[r] & X_s^2 \cM_2 \ar[r] & X_s^3 \cM_2  &\ldots\\
&&&\cM_3 \ar[r] & X_s \cM_3 \ar[r] & X_s^2 \cM_3  &\ldots\\
&&&&\cM_4 \ar[r] & X_s \cM_4   &\ldots\\
&&&&&\cM_5&\ldots
}
\]
where all summands are simple weight $\mathfrak{sp}(2n,\mR)$-modules. 
The $k$-th column gives the decomposition of the space of homogeneous 
polynomials of degree 
$k$ taking values in $\cC = L(-\frac{1}{2}\omega_n)\oplus L(\omega_{n-1}-\frac{3}{2}\omega_n)$. 
The $l$-th row forms a lowest weight 
$\mathfrak{sl}(2,\mR)$-module $\oplus_{j=0}^\infty X_s^j{\fam2 M}_l$ generated by the space 
of symplectic monogenics ${\fam2 M}_l$.
\end{thm}
Notice that beside the unitary Segal-Shale-Weil representation, all remaining 
simple weight modules ${\fam2 M}_l$, $l=1,2,\dots$, appearing in the Fischer decomposition, are
not unitarizable. This is a direct consequence of the classification of unitarizable highest weight modules, \cite{ehw}.

One immediate Corollary is the structure of polynomial solutions of the symplectic Dirac operator 
on $\mR^{2n}$. The statement is given for both symplectic spin modules
$L(-\frac{1}{2}\omega_n)$ and $L(\omega_{n-1}-\frac{3}{2}\omega_n)$ separately.
\begin{cor}
\label{symkernel}
The kernel of (half of) the symplectic Dirac operator $D_s$ acting 
on $L(-\frac{1}{2}\omega_n)$-valued polynomials is 
\begin{eqnarray*} 
\mathrm{Ker}^+(D_s)\simeq \bigoplus_{l\in\mN_0}
\left(L(2l\omega_1-\frac{1}{2}\omega_n)\oplus L((2l+1)\omega_1-\frac{1}{2}\omega_n)\right).
\end{eqnarray*}
The kernel of (half of) the symplectic Dirac operator $D_s$ acting 
on $L(\omega_{n-1}-\frac{3}{2}\omega_n)$-valued polynomials is 
\begin{eqnarray*} 
\mathrm{Ker}^-(D_s)\simeq \bigoplus_{l\in\mN_0} \left(
L(2l\omega_1+\omega_{n-1}-\frac{3}{2}\omega_n)\oplus L((2l+1)\omega_1+\omega_{n-1}-\frac{3}{2}\omega_n)\right).
\end{eqnarray*}
\end{cor}


Every homogeneous polynomial of degree $k$ valued in $\cC$, can now 
be decomposed into monogenic components as follows.
\begin{thm}
\label{thmReprHomK}
Let $p\in \cP_{k} \otimes \cC$.  Then there exists a unique representation of $p$ as
\[
p=\sum_{i=0}^k p_i,
\]
where $p_i=X_s^{k-i}m_i$ and $m_i\in {\fam2 M}_i$. 
\end{thm}

We now proceed to construct projection operators that allow to explicitly compute 
the decomposition given in Theorem \ref{thmReprHomK}. They are given in the following theorem.

\begin{thm}
The operators
\begin{equation}
\pi^{k}_{i} = \sum_{j=0}^{k-i} a^{i,k}_{j} X^{i+j}_{s} D^{i+j}_{s}\in\mathrm{End}(\cP \otimes \cC)
\label{projversion1}
\end{equation}
with
\[
a^{i,k}_{j} = (-1)^{j} (2n+2k-2i-1) \frac{2^{i+j}}{i! j!} \frac{(2n+2k-2i-j-2)!}{(2n+2k-i-1)}
\]
and $i=0, \ldots, k$ satisfy
\[
\pi^{k}_{i} (X^{j}_{s} \cM_{k-j}) = \delta_{ij} X^{i}_{s} \cM_{k-i}.
\]
\end{thm}

\begin{proof}
Using Lemma \ref{calcSympDirac} it is easy to see that $\pi^{k}_{i} (X^{j}_{s} \cM_{k-j}) =0$ for all $j<i$. The coefficients $a^{i,k}_{j}$, for fixed $i$ and $k$, can now be determined iteratively. First of all, expressing $\pi^{k}_{i} (X^{i}_{s} \cM_{k-i}) = X^{i}_{s} \cM_{k-i}$ yields
\[
a^{i,k}_{0} = \frac{1}{c_{i,i,k-i}} = \frac{2^{i}}{i!} \frac{(2n+2k-2i-1)!}{(2n+2k-i-1)!}.
\]
Similarly, $\pi^{k}_{i} (X^{i+1}_{s} \cM_{k-i-1}) =0$ yields
\[
a^{i,k}_{1} = - \frac{c_{i,i+1,k-i-1}}{c_{i+1,i+1,k-i-1}} a^{i,k}_{0} = - \frac{1}{n+k-i-1} a^{i,k}_{0}.
\]
Thus continuing we arrive at the hypothesis
\[
a^{i,k}_{j} = (-1)^{j} \frac{2^{j}}{j!} \frac{(2n+2k-2i-j-2)!}{(2n+2k-2i-2)!} a^{i,k}_{0},
\]
which can be proven using induction. Indeed, suppose that the statement holds for $a_j^{i,k}$, $j \leq l$, then we prove that it also holds for $a_{l+1}^{i,k}$. This last coefficient has to satisfy
\[
\sum_{j=0}^{l+1} a_j^{i,k} \, c_{i+j,i+l+1,k-i-l-1}=0.
\]
Substituting the known expressions we obtain
\begin{align*}
a_{l+1}^{i,k} &= -\sum_{j=0}^l a_j^{i,k} \frac{c_{i+j,i+l+1,k-i-l-1}}{c_{i+l+1,i+l+1,k-i-l-1}}\\
&= -\sum_{j=0}^l a_j^{i,k} \frac{2^{l+1-j}}{(l+1-j)!} \frac{(\alpha-2l-1)!}{(\alpha-l-j)}\\
&= -\frac{2^{l+1}}{(l+1)!} \frac{(\alpha-2l-1)!}{\alpha!} a^{i,k}_{0}  \sum_{j=0}^{l} (-1)^{j} \binom{l+1}{j} \frac{(\alpha-j)!}{(\alpha-l-j)!},
\end{align*}
where we have put $\alpha = 2n+2k-2i-2$. The proof is now complete by remarking that
\[
\sum_{j=0}^{l+1} (-1)^{j} \binom{l+1}{j} \frac{(\alpha-j)!}{(\alpha-l-j)!}=0.
\]
This can either be obtained directly (see e.g. Lemma 5 in \cite{HDBThesis}) or as a consequence of Gauss's hypergeometric theorem, expressing $\,_2F_1(a,b;c;1)$ in terms of a product of Gamma functions.
\end{proof}

Note that there exists another way of computing the projection operators
on irreducible summands, namely using the Casimir operator of $\mathfrak{sl}(2,\mR)$. First observe that 
\[
\Gamma_{s} \cM_{k} = -\frac{1}{2} k (2n-1+k) \cM_{k}.
\]
It is then clear that the operators
\[
\mP_{i}^{k} =  \prod_{j=0, j \neq i}^{k} \dfrac{2\Gamma_{s} + j (2n-1+j)}{j(2n-1+j)-i(2n-1+i)}
\quad\in\,\mathrm{End}(\cP \otimes \cC)
\]
for $i = 0, \ldots, k$ satisfy
\[
\mP_{i}^{k} (X_{s}^{k-j} \cM_{j}) = \delta_{i j}X_{s}^{k-i} \cM_{i}.
\]

\section{Open questions and unresolved problems}

In \cite{KAD}, the symplectic Dirac operator $D_s$ on 
$\mR^{2n}\hookrightarrow \mR^{2n+1}$ was studied as an
$\mathfrak{sp}(2n+2,\mR)$-equivariant differential operator 
in the context of the contact parabolic geometry on 
the big open cell of a homogeneous 
space $\mR^{2n+1}\hookrightarrow Sp(2n+2,\mR)/P$ for the 
maximal parabolic subgroup $P\subset Sp(2n+2,\mR)$ with the 
nilpotent subgroup of $P$ isomorphic to the Heisenberg group.
As a consequence, the kernel of $D_s$ has the structure of 
an $\mathfrak{sp}(2n+2,\mR)$-module and we leave the question 
of its representation theoretic content open.

In \cite{Orsted}, the authors studied a particular deformation 
of the Howe duality and Fischer 
decomposition for the Dirac operator acting on spinor valued polynomials, coming from the Dunkl
deformation of the Dirac operator. It is an interesting question to develop the Dunkl version
of the symplectic Dirac operator in the context of symplectic reflection algebras (see \cite{MR1881922}).  

In the context of Lie superalgebras, the symplectic Dirac operator may be viewed as an operator that combines with the Dirac operator of \cite{slup} to form a ``twisted'' super Dirac operator  (see bottom half of Figure 1 of \cite{C}). A detailed representation theoretic study of this operator would complete the picture of \cite{C}.

We should also remark that analytic properties of symplectic monogenics were not studied at all,
and it is challenging to employ the techniques of e.g., the symplectic Fourier transform, to understand
the properties of the symplectic Dirac operator $D_s$ in detail.  

\hspace{0.5cm}

{\bf Acknowledgement:} P.S. and V.S. are supported by GA CR P201/12/G028.

\vspace{0.1cm}
\flushleft{Hendrik De Bie\\      
Department of Mathematical Analysis,\\ Ghent University, Galglaan 2, 9000 Gent, Belgium.\\            
Hendrik.DeBie@UGent.be               
}

\vspace{0.1cm}
{Petr Somberg\\               
Mathematical Institute of Charles University,\\ Sokolovsk\'a 83, 186 75 Praha, Czech Republic.\\            
somberg@karlin.mff.cuni.cz              
}

\vspace{0.1cm}
{
Vladimir Sou{\v{c}}ek\\               
Mathematical Institute of Charles University,\\ Sokolovsk\'a 83, 186 75 Praha, Czech Republic.\\            
soucek@karlin.mff.cuni.cz  
}

\end{document}